\documentclass[10pt]{article}
\usepackage{amsfonts}
\usepackage{amsmath}
\usepackage{amssymb}
\usepackage{amsthm}
\usepackage{enumerate}
\makeatletter
\def\imod#1{\allowbreak\mkern10mu({\operator@font mod}\,\,#1)}
\makeatother

\usepackage{setspace}

\newtheorem{theorem}{Theorem}[section]
\newtheorem{lemma}{Lemma}[section]
\newtheorem{corollary}{Corollary}[section]
\newtheorem{fact}{Fact}[section]

\theoremstyle{definition}
\newtheorem{definition}{Definition}[section]
\newtheorem{remark}{Remark}[section]

\begin{document}
\begin{center}
\begin{singlespace}
\vskip 1cm{\LARGE\bf Multiperfect Numbers in Certain Quadratic Rings
\vskip 1cm
\large
Colin Defant\footnote{This work was supported by National Science Foundation grant no. 1262930.}\footnote{Colin Defant\\ 18434 Hancock Bluff Rd. \\ Dade City, FL 33523}\\
Department of Mathematics\\
University of Florida\\
United States\\
cdefant@ufl.edu}
\end{singlespace}
\end{center}
\vskip .2 in

\begin{abstract}
Using an extension of the abundancy index to imaginary quadratic rings that are unique factorization domains, we investigate what we call $n$-powerfully $t$-perfect numbers in these rings. This definition serves to extend the concept of multiperfect numbers that have been defined and studied in the integers. At the end of the paper, as well as at various points throughout the paper, we point to some potential areas for further research.    
\end{abstract}

\section{Introduction} 

Throughout this paper, we will let $\mathbb{N}$ denote the set of positive integers, and we will let $\mathbb{P}=\{2, 3, 5, \ldots\}$ denote the set of (integer) prime numbers. 
\par 
The arithmetic functions $\sigma_k$ are defined, for every integer $k$, by  \\ 
$\displaystyle{\sigma_k(n)=\sum_{\substack{c\vert n\\c>0}}c^k}$. For each integer $k\neq 0$, $\sigma_k$ is multiplicative and satisfies \\ 
$\displaystyle{\sigma_k (p^\alpha)=\frac{p^{k(\alpha+1)}-1}{p^k-1}}$ for all (integer) primes $p$ and positive integers $\alpha$. The abundancy index of a positive integer $n$ is defined by $\displaystyle{I(n)=\frac{\sigma_1(n)}{n}}$. Some of the most interesting questions related to the abundancy index are those dealing with perfect and multiperfect numbers. 
\par
A positive integer $n$ is said to be $t$-perfect if $I(n)=t$ for a positive integer $t\geq 2$, and $2$-perfect numbers, which have been studied since the ancient Greeks, are called perfect numbers. It is known that there is a one-to-one correspondence between even perfect numbers and Mersenne primes, so it is, therefore, unknown whether or not there are infinitely many even perfect numbers. Although no odd perfect numbers are currently known to exist, a long list of criteria, sometimes known as Sylvester's Web of Conditions, places demands on the properties that any odd perfect number would need to satisfy.     
\par 
For any square-free integer $d$, let $\mathcal O_{\mathbb{Q}(\sqrt{d})}$ be the quadratic integer ring given by \[\mathcal O_{\mathbb{Q}(\sqrt{d})}=\begin{cases} \mathbb{Z}[\frac{1+\sqrt{d}}{2}], & \mbox{if } d\equiv 1\imod{4}; \\ \mathbb{Z}[\sqrt{d}], & \mbox{if } d\equiv 2, 3 \imod{4}. \end{cases}\] 
\par 
Throughout the remainder of this paper, we will work in the rings $\mathcal O_{\mathbb{Q}(\sqrt{d})}$ for different specific or arbitrary values of $d$. We will use the symbol ``$\vert$" to mean ``divides" in the ring $\mathcal O_{\mathbb{Q}(\sqrt{d})}$ in which we are working. 
Whenever we are working in a ring other than $\mathbb{Z}$, we will make sure to emphasize when we wish to state that one integer divides another in $\mathbb{Z}$. 
For example, if we are working in $\mathbb{Z}[i]$, the ring of Gaussian integers, we might say that $1+i\vert 1+3i$ and that $2\vert 6$ in $\mathbb{Z}$. We will also refer to primes in $\mathcal O_{\mathbb{Q}(\sqrt{d})}$ as ``primes," whereas we will refer to (positive) primes in $\mathbb{Z}$ as ``integer primes." For an integer prime $p$ and a nonzero integer $n$, we will let $\upsilon_p(n)$ denote the largest integer $k$ such that $p^k\vert n$ in $\mathbb{Z}$. For a prime $\pi$ and a nonzero number $x\!\in\!\mathcal O_{\mathbb{Q}(\sqrt{d})}$, we will let $\rho_\pi(x)$ denote the largest integer $k$ such that $\pi^k\vert x$.  
Furthermore, we will henceforth focus exclusively on values of $d$ for which $\mathcal O_{\mathbb{Q}(\sqrt{d})}$ is a unique factorization domain and $d<0$. In other words, $d\in K$, where we will define $K$ to be the set $\{-163,-67,-43,-19,-11,-7,-3,-2,-1\}$. The set $K$ is known to be the complete set of negative values of $d$ for which $\mathcal O_{\mathbb{Q}(\sqrt{d})}$ is a unique factorization domain \cite{Stark67}.  
\par 
For now, let us work in a general ring $\mathcal O_{\mathbb{Q}(\sqrt{d})}$ such that $d\!\in\! K$. For an element $a+b\sqrt{d}\in\mathcal O_{\mathbb{Q}(\sqrt{d})}$ with $a,b\in \mathbb{Q}$, we define the conjugate by $\overline{a+b\sqrt{d}}=a-b\sqrt{d}$. We also define the norm of an element $z$ by $N(z)=z\overline{z}$ and the absolute value of $z$ by $\vert z\vert=\sqrt{N(z)}$. We assume familiarity with the properties of these object, which are treated in Keith Conrad's online notes \cite{Conrad}. For $x,y\in\mathcal O_{\mathbb{Q}(\sqrt{d})}$, we say that $x$ and $y$ are associated, denoted $x\sim y$, if and only if $x=uy$ for some unit $u$ in the ring $\mathcal O_{\mathbb{Q}(\sqrt{d})}$. Furthermore, we will make repeated use of the following well-known facts. 
\begin{fact} \label{Fact1.1}
Let $d\!\in\! K$. If $p$ is an integer prime, then exactly one of the following is true. 
\begin{itemize}
\item $p$ is also a prime in $\mathcal O_{\mathbb{Q}(\sqrt{d})}$. In this case, we say that $p$ is inert in $\mathcal O_{\mathbb{Q}(\sqrt{d})}$. 
\item $p\sim \pi^2$ and $\pi\sim\overline{\pi}$ for some prime $\pi\in \mathcal O_{\mathbb{Q}(\sqrt{d})}$. In this case, we say $p$ ramifies (or $p$ is ramified) in $\mathcal O_{\mathbb{Q}(\sqrt{d})}$. 
\item $p=\pi\overline{\pi}$ and $\pi\not\sim\overline{\pi}$ for some prime $\pi\in\mathcal O_{\mathbb{Q}(\sqrt{d})}$. In this case, we say $p$ splits (or $p$ is split) in $\mathcal O_{\mathbb{Q}(\sqrt{d})}$.
\end{itemize}
\end{fact}
\begin{fact} \label{Fact1.2}
Let $d\!\in\! K$. If $\pi\!\in\!\mathcal O_{\mathbb{Q}(\sqrt{d})}$ is a prime, then exactly one of the following is true. 
\begin{itemize}
\item $\pi\sim q$ and $N(\pi)=q^2$ for some inert integer prime $q$. 
\item $\pi\sim\overline{\pi}$ and $N(\pi)=p$ for some ramified integer prime $p$. 
\item $\pi\not\sim\overline{\pi}$ and $N(\pi)=N(\overline{\pi})=p$ for some split integer prime $p$. 
\end{itemize}
\end{fact}
\begin{fact} \label{Fact1.3}
If $d\!\in K\!$, $q$ is an integer prime that is inert in $\mathcal O_{\mathbb{Q}(\sqrt{d})}$, and $x\in\mathcal O_{\mathbb{Q}(\sqrt{d})}\backslash\{0\}$, then $\upsilon_q(N(x))$ is even and $\rho_q(x)=\frac{1}{2}\upsilon_q(N(x))$.
\end{fact}
\begin{fact} \label{Fact1.4}
Let $p$ be an odd integer prime. Then $p$ ramifies in $\mathcal O_{\mathbb{Q}(\sqrt{d})}$ if and only if $p\vert d$ in $\mathbb{Z}$. If $p\nmid d$ in $\mathbb{Z}$, then $p$ splits in $\mathcal O_{\mathbb{Q}(\sqrt{d})}$ if and only if $d$ is a quadratic residue modulo $p$. Note that this implies that $p$ is inert in $\mathcal O_{\mathbb{Q}(\sqrt{d})}$ if and only if $p\nmid d$ in $\mathbb{Z}$ and $d$ is a quadratic nonresidue modulo $p$. 
Also, the integer prime $2$ ramifies in $\mathcal O_{\mathbb{Q}(\sqrt{-1})}$ and $\mathcal O_{\mathbb{Q}(\sqrt{-2})}$, splits in $\mathcal O_{\mathbb{Q}(\sqrt{-7})}$, and is inert in $\mathcal O_{\mathbb{Q}(\sqrt{d})}$ for all $d\in K\backslash\{-1,-2,-7\}$.
\end{fact}
\begin{fact} \label{Fact1.5}
Let $\mathcal O_{\mathbb{Q}(\sqrt{d})}^*$ be the set of units in the ring $\mathcal O_{\mathbb{Q}(\sqrt{d})}$. Then $\mathcal O_{\mathbb{Q}(\sqrt{-1})}^*=\{\pm 1,\pm i\}$, $\displaystyle{\mathcal O_{\mathbb{Q}(\sqrt{-3})}^*=\left\{\pm 1,\pm \frac{1+\sqrt{-3}}{2},\pm \frac{1-\sqrt{-3}}{2}\right\}}$, and $\mathcal O_{\mathbb{Q}(\sqrt{d})}^*=\{\pm 1\}$ \\ 
whenever $d\in K\backslash\{-1,-3\}$. 
\end{fact}
\par 
For a nonzero complex number $z$, let $\arg (z)$ denote the argument, or angle, of $z$. We convene to write $\arg (z)\in [0,2\pi)$ for all $z\in\mathbb{C}$. For each $d\in K$, we define the set $A(d)$ by 
\[A(d)=\begin{cases} \{z\in\mathcal O_{\mathbb{Q}(\sqrt{d})} \backslash\{0\}: 0\leq \arg (z)<\frac{\pi}
{2}\}, & \mbox{if } d=-1; \\ \{z\in\mathcal O_{\mathbb{Q}(\sqrt{d})} \backslash\{0\}: 0\leq \arg (z)<\frac{\pi}
{3}\}, & \mbox{if } d=-3; \\ \{z\in\mathcal O_{\mathbb{Q}(\sqrt{d})} \backslash\{0\}: 0\leq \arg (z)<\pi\}, & \mbox{otherwise}. \end{cases}\] 
Thus, every nonzero element of $\mathcal O_{\mathbb{Q}(\sqrt{d})}$ can be written uniquely as a unit times a product of primes in $A(d)$. Also, every $z\in\mathcal O_{\mathbb{Q}(\sqrt{d})}\backslash\{0\}$ is associated to a unique element of $A(d)$. The author has defined analogues of the arithmetic functions $\sigma_k$ in quadratic rings $\mathcal O_{\mathbb{Q}(\sqrt{d})}$ with $d\in K$ \cite{Defant14}, and we will state the important definitions and properties for the sake of completeness.   
\begin{definition} \label{Def1.1}
Let $d\in K$, and let $n\in \mathbb{Z}$. 
Define the function 
\newline $\delta_n\colon\mathcal O_{\mathbb{Q}(\sqrt{d})}\backslash\{0\}\rightarrow [1,\infty)$ by 
\[\delta_n (z)=\sum_{\substack{x\vert z\\x\in A(d)}}\vert x \vert^n.\]
\end{definition}
\begin{remark} \label{Rem1.1}
We note that, for each $x$ in the summation in the above definition, we may cavalierly replace $x$ with one of its associates. This is because associated numbers have the same absolute value. In other words, the only reason for the criterion $x\!\in\! A (d)$ in the summation that appears in Definition \ref{Def1.1} is to forbid us from counting associated divisors as distinct terms in the summation, but we may choose to use any of the associated divisors as long as we only choose one. This should not be confused with how we count conjugate divisors (we treat $2+i$ and $2-i$ as distinct divisors of $5$ in $\mathbb{Z}[i]$ because $2+i\not\sim 2-i$).  
\end{remark}

\begin{remark} \label{Rem1.2}
We mention that the function $\delta_n$ is different in each ring $\mathcal O_{\mathbb{Q}(\sqrt{d})}$. Perhaps it would be more precise to write $\delta_n(z,d)$, but we will omit the latter component for convenience. We note that we will also use this convention with functions such as $I_n$ (which we will define soon). 
\end{remark}
\par 
We will say that a function $f\colon\mathcal O_{\mathbb{Q}(\sqrt{d})}\backslash\{0\}\!\rightarrow\!\mathbb{R}$ is multiplicative if $f(xy)=f(x)f(y)$ whenever $x$ and $y$ are relatively prime (have no nonunit common divisors). The author has shown that,
for any integer $n$, $\delta_n$ is multiplicative \cite{Defant14}.
\begin{definition} \label{Def1.2}
For each positive integer $n$, define the function \\ 
$I_n\colon\mathcal O_{\mathbb{Q}(\sqrt{d})}\backslash\{0\}\rightarrow[1,\infty)$ by $\displaystyle{I_n(z)=\frac{\delta_n(z)}{\vert z\vert ^n}}$. 
For a positive integer $t\geq 2$, we say that a number $z\!\in\!\mathcal O_{\mathbb{Q}(\sqrt{d})}\backslash\{0\}$ is \textit{$n$-powerfully $t$-perfect in $\mathcal O_{\mathbb{Q}(\sqrt{d})}$} if $I_n(z)=t$, and, if $t=2$, we simply say that $z$ is \textit{$n$-powerfully perfect in $\mathcal O_{\mathbb{Q}(\sqrt{d})}$}. 
Whenever $n=1$, we will omit the adjective ``$1$-powerfully." 
\end{definition} 
As an example, we will let $d=-1$ so that $\mathcal O_{\mathbb{Q}(\sqrt{d})}=\mathbb{Z}[i]$. Let us compute $I_2(9+3i)$. We have $9+3i=3(1+i)(2-i)$, so $\delta_2(9+3i)=N(1)+N(3)+N(1+i)+N(2-i)+N(3(1+i))+N(3(2-i))+N((1+i)(2-i))+N(3(1+i)(2-i))=1+9+2+5+18+45+10+90=180$. Then $\displaystyle{I_2(9+3i)=\frac{180}{N(3(1+i)(2-i))}=2}$, so $9+3i$ is $2$-powerfully perfect in $\mathcal O_{\mathbb{Q}(\sqrt{-1})}$.  
\par 
We omit the (fairly simplistic) proof of the following theorem because it is included in \cite{Defant14}. 
\begin{theorem} \label{Thm1.1}
Let $n\!\in\!\mathbb{N}$, $d\!\in\! K$, and $z_1, z_2, \pi\in\mathcal O_{\mathbb{Q}(\sqrt{d})}\backslash\{0\}$ with $\pi$ a prime. Then, if we are working in the ring $\mathcal O_{\mathbb{Q}(\sqrt{d})}$, the following statements are true. 
\begin{enumerate}[(a)] % (a), (b), (c), (d)
\item The range of $I_n$ is a subset of the interval $[1,\infty)$, and $I_n(z_1)=1$ if and only if $z_1$ is a unit in $\mathcal O_{\mathbb{Q}(\sqrt{d})}$. If $n$ is even, then $I_n(z_1)\in\mathbb{Q}$. 
\item $I_n$ is multiplicative.  
\item $I_n(z_1)=\delta_{-n}(z_1)$. 
\item If $z_1\vert z_2$, then $I_n(z_1)\leq I_n(z_2)$, with equality if and only if $z_1\sim z_2$. 
\end{enumerate}  	
\end{theorem} 
Henceforth, we will focus on the existence of $n$-powerfully $t$-perfect numbers for $n\neq 2$. 

\section{Exploring $n$-powerfully $t$-perfect Numbers for $n\neq 2$} 
We begin this section with a theorem (after two short lemmata) that dramatically limits the number of possibilities that we may consider when dealing with $n$-powerfully $t$-perfect numbers. 
\begin{lemma} \label{Lem2.1}
Let $d\in K$, and let $n\in\mathbb{N}$. If $z\in\mathcal O_{\mathbb{Q}(\sqrt{d})}\backslash\{0\}$ and $n\geq 3$, then $\displaystyle{I_n(z)<\zeta\left(\frac{n}{2}\right)^2}$, where $\zeta$ denotes the Riemann zeta function.   
\end{lemma}
\begin{proof}
Let $\Psi(z)$ be the set of primes in $A(d)$ that divide $z$, and let $\Phi$ be the set of primes in $A(d)$. By parts $(b)$ and $(c)$ of Theorem \ref{Thm1.1}, as well as Fact \ref{Fact1.2}, we may write 
\[I_n(z)=\prod_{\pi\in\Psi(z)}\left(\sum_{j=0}^{\rho_\pi(z)}\frac{1}{\vert\pi^j\vert^n}\right)\]
\[=\prod_{\substack{\pi\in\Psi(z)\\ \vert\pi\vert\in\mathbb{N}}}\left(\sum_{j=0}^{\rho_\pi(z)}\frac{1}{\vert\pi^j\vert^n}\right)  \prod_{\substack{\pi\in\Psi(z)\\ \vert\pi\vert\not\in\mathbb{N} \\ \pi\sim\overline{\pi}}}\left(\sum_{j=0}^{\rho_\pi(z)}\frac{1}{\vert\pi^j\vert^n}\right) \prod_{\substack{\pi\in\Psi(z)\\ \vert\pi\vert\not\in\mathbb{N} \\ \pi\not\sim\overline{\pi}}}\left(\sum_{j=0}^{\rho_\pi(z)}\frac{1}{\vert\pi^j\vert^n}\right)\]
\[<\prod_{\substack{\pi\in\Phi\\ \vert\pi\vert\in\mathbb{N}}}\left(\sum_{j=0}^{\infty}\frac{1}{\vert\pi^j\vert^n}\right)  \prod_{\substack{\pi\in\Phi\\ \vert\pi\vert\not\in\mathbb{N} \\ \pi\sim\overline{\pi}}}\left(\sum_{j=0}^{\infty}\frac{1}{\vert\pi^j\vert^n}\right) \prod_{\substack{\pi\in\Phi\\ \vert\pi\vert\not\in\mathbb{N} \\ \pi\not\sim\overline{\pi}}}\left(\sum_{j=0}^{\infty}\frac{1}{\vert\pi^j\vert^n}\right)\]
\[=\prod_{\substack{q\in\mathbb{P}\\ q\hspace{0.75 mm} is\hspace{0.75 mm} inert}}\left(\sum_{j=0}^{\infty}\frac{1}{q^{jn}}\right)
\prod_{\substack{p\in\mathbb{P}\\ p\hspace{0.75 mm} ramifies}}\left(\sum_{j=0}^{\infty}\frac{1}{\sqrt{p}^{jn}}\right)
\prod_{\substack{p\in\mathbb{P}\\ p\hspace{0.75 mm} splits}}\left(\sum_{j=0}^{\infty}\frac{1}{\sqrt{p}^{jn}}\right)^2\]
\[<\prod_{\substack{q\in\mathbb{P}\\ q\hspace{0.75 mm} is\hspace{0.75 mm} inert}}\left(\sum_{j=0}^{\infty}\frac{1}{\sqrt{q}^{jn}}\right)^2
\prod_{\substack{p\in\mathbb{P}\\ p\hspace{0.75 mm} ramifies}}\left(\sum_{j=0}^{\infty}\frac{1}{\sqrt{p}^{jn}}\right)^2
\prod_{\substack{p\in\mathbb{P}\\ p\hspace{0.75 mm} splits}}\left(\sum_{j=0}^{\infty}\frac{1}{\sqrt{p}^{jn}}\right)^2\]
\[=\prod_{p\in\mathbb{P}}\left(\sum_{j=0}^{\infty}\frac{1}{\sqrt{p}^{jn}}\right)^2=\zeta\left(\frac{n}{2}\right)^2.\]
\end{proof} 
We state the following lemma without proof, though the proof may be found as a corollary of Lemma 3.4 in \cite{Defant14}. 
\begin{lemma} \label{Lem2.2}
Let us fix $d\in K$ and work in the ring $\mathcal O_{\mathbb{Q}(\sqrt{d})}$. Let $n$ be an odd positive integer, and let $z\in\mathcal O_{\mathbb{Q}(\sqrt{d})}\backslash\{0\}$. If $I_n(z)$ is rational, then all primes dividing $z$ are associated to inert integer primes.  
\end{lemma} 
\begin{theorem} \label{Thm2.1}
Let $d\in K$. For any integers $n\geq 3$ and $t\geq 2$, there are no $n$-powerfully $t$-perfect numbers in $\mathcal O_{\mathbb{Q}(\sqrt{d})}$ because $I_n(z)<2$ whenever $z\in\mathcal O_{\mathbb{Q}(\sqrt{d})}\backslash\{0\}$ and $I_n(z)\in\mathbb{Q}$. 
\end{theorem}
\begin{proof}
Let $z\in\mathcal O_{\mathbb{Q}(\sqrt{d})}\backslash\{0\}$ be such that $I_n(z)\in\mathbb{Q}$. If $n\geq 5$, the proof follows from Lemma \ref{Lem2.1} because, in that case, we have 
$\displaystyle{I_n(z)<\zeta\left(\frac{5}{2}\right)^2\approx 1.799<2}$.
Now, let $n=3$. Because $n$ is an odd positive integer and $I_n(z)$ is rational, Lemma \ref{Lem2.2} tells us that any prime dividing $z$
must be associated to an inert integer prime. Therefore, 
\[I_n(z)=\prod_{\substack{\pi\vert z \\\ \pi\in A(d)}}\left(\sum_{j=0}^{\rho_\pi(z)}\frac{1}{\vert\pi^j\vert^3}\right)<\prod_{\substack{q\in\mathbb{P} \\ q\hspace{0.75 mm} is\hspace{0.75 mm} inert}}\left(\sum_{j=0}^\infty\frac{1}{q^{3j}}\right)\]
\[<\prod_{q\in\mathbb{P}}\left(\sum_{j=0}^\infty\frac{1}{q^{3j}}\right)=\zeta(3)<2.\]
\par 
The only case left to consider is the case $n=4$. As an intermediate step in the proof of Lemma \ref{Lem2.1}, we arrived at the inequality 
\begin{equation} \label{Eq2.1}
I_n(z)<\prod_{\substack{q\in\mathbb{P}\\ q\hspace{0.75 mm} is\hspace{0.75 mm} inert}}\left(\sum_{j=0}^{\infty}\frac{1}{q^{jn}}\right)
\prod_{\substack{p\in\mathbb{P}\\ p\hspace{0.75 mm} ramifies}}\left(\sum_{j=0}^{\infty}\frac{1}{\sqrt{p}^{jn}}\right)
\prod_{\substack{p\in\mathbb{P}\\ p\hspace{0.75 mm} splits}}\left(\sum_{j=0}^{\infty}\frac{1}{\sqrt{p}^{jn}}\right)^2.
\end{equation}
Substituting $n=4$, we have 
\[I_4(z)<\prod_{\substack{q\in\mathbb{P}\\ q\hspace{0.75 mm} is\hspace{0.75 mm} inert}}\left(\sum_{j=0}^{\infty}\frac{1}{q^{4j}}\right)
\prod_{\substack{p\in\mathbb{P}\\ p\hspace{0.75 mm} ramifies}}\left(\sum_{j=0}^{\infty}\frac{1}{p^{2j}}\right)
\prod_{\substack{p\in\mathbb{P}\\ p\hspace{0.75 mm} splits}}\left(\sum_{j=0}^{\infty}\frac{1}{p^{2j}}\right)^2.\]
Now, suppose that the ring $\mathcal O_{\mathbb{Q}(\sqrt{d})}$ in which we are working is one in which the integer prime $2$ is inert. Then 
\[I_4(z)<\left(\sum_{j=0}^\infty\frac{1}{2^{4j}}\right)\prod_{\substack{q\in\mathbb{P}\\ q\hspace{0.50 mm} is\hspace{0.50 mm} inert \\ q\neq 2}}\left(\sum_{j=0}^{\infty}\frac{1}{q^{4j}}\right)
\prod_{\substack{p\in\mathbb{P}\\ p\hspace{0.75 mm} ramifies}}\left(\sum_{j=0}^{\infty}\frac{1}{p^{2j}}\right)
\prod_{\substack{p\in\mathbb{P}\\ p\hspace{0.75 mm} splits}}\left(\sum_{j=0}^{\infty}\frac{1}{p^{2j}}\right)^2\] 
\[\leq\left(\sum_{j=0}^\infty\frac{1}{2^{4j}}\right)\prod_{\substack{q\in\mathbb{P}\\ q\hspace{0.75 mm} is\hspace{0.75 mm} inert \\ q\neq 2}}\left(\sum_{j=0}^{\infty}\frac{1}{q^{2j}}\right)^2
\prod_{\substack{p\in\mathbb{P}\\ p\hspace{0.75 mm} ramifies}}\left(\sum_{j=0}^{\infty}\frac{1}{p^{2j}}\right)^2
\prod_{\substack{p\in\mathbb{P}\\ p\hspace{0.75 mm} splits}}\left(\sum_{j=0}^{\infty}\frac{1}{p^{2j}}\right)^2\] 
\[=\left(\sum_{j=0}^\infty\frac{1}{2^{4j}}\right)\prod_{\substack{p\in\mathbb{P}\\ p\neq 2}}\left(\sum_{j=0}^\infty\frac{1}{p^{2j}}\right)^2 =\zeta(2)^2\left(\sum_{j=0}^\infty\frac{1}{2^{2j}}\right)^{-2}\left(\sum_{j=0}^\infty\frac{1}{2^{4j}}\right)\]
\[=\frac{\pi^4}{36}\cdot\frac{9}{16}\cdot\frac{16}{15}=\frac{\pi^4}{60}<2.\]
Now, recall from Fact \ref{Fact1.4} that the only $d\in K$ for which $2$ is not inert in $\mathcal O_{\mathbb{Q}(\sqrt{d})}$ are $d=-1$, $d=-2$, and $d=-7$. If $d=-1$, then $2$ ramifies and $3$ is inert. Therefore, if we write $\displaystyle{\mathcal H=\left(\sum_{j=0}^\infty\frac{1}{2^{2j}}\right)\left(\sum_{j=0}^\infty\frac{1}{3^{4j}}\right)}$, then
\[I_4(z)<\mathcal H\prod_{\substack{q\in\mathbb{P}\\ q\hspace{0.75 mm} is\hspace{0.75 mm} inert \\ q\neq 3}}\left(\sum_{j=0}^{\infty}\frac{1}{q^{4j}}\right)
\prod_{\substack{p\in\mathbb{P}\\ p\hspace{0.75 mm} ramifies \\ p\neq 2}}\left(\sum_{j=0}^{\infty}\frac{1}{p^{2j}}\right)
\prod_{\substack{p\in\mathbb{P}\\ p\hspace{0.75 mm} splits}}\left(\sum_{j=0}^{\infty}\frac{1}{p^{2j}}\right)^2\]
\[\leq\mathcal H\prod_{\substack{p\in\mathbb{P} \\ p\not\in
\{2, 3\}}}\left(\sum_{j=0}^\infty\frac{1}{p^{2j}}\right)^2=\zeta(2)^2\left(\sum_{j=0}^\infty\frac{1}{2^{2j}}\right)^{-1}\left(\sum_{j=0}^\infty
\frac{1}{3^{2j}}\right)^{-2}\left(\sum_{j=0}^\infty\frac{1}{3^{4j}}\right)\]
\[=\frac{\pi^4}{36}\cdot
\frac{3}{4}\cdot\frac{64}{81}\cdot\frac{81}{80}=\frac{\pi^4}{60}<2.\]
Similarly, if $d=-2$, then $2$ ramifies and $5$ is inert. Therefore, we may replace all of the $3$'s in the above chain of inequalities with $5$'s to arrive at
\[I_4(z)<\zeta(2)^2\left(\sum_{j=0}^\infty\frac{1}{2^{2j}}\right)^{-1}\left(\sum_{j=0}^\infty
\frac{1}{5^{2j}}\right)^{-2}\left(\sum_{j=0}^\infty\frac{1}{5^{4j}}\right)\]
\[=\frac{\pi^4}{36}\cdot
\frac{3}{4}\cdot\frac{576}{625}\cdot\frac{625}{624}=\frac{\pi^4}{52}<2.\]
Finally, we consider the case $d=-7$. In this case, $3$ and $5$ are both inert. If we write $\displaystyle{\mathcal J=\left(\sum_{j=0}^\infty\frac{1}{3^{4j}}\right)\left(\sum_{j=0}^\infty\frac{1}{5^{4j}}\right)}$, then
\[I_4(z)<\mathcal J\prod_{\substack{q\in\mathbb{P}\\ q\hspace{0.75 mm} is\hspace{0.75 mm} inert \\ q\not\in\{3, 5\}}}\left(\sum_{j=0}^{\infty}\frac{1}{q^{4j}}\right)
\prod_{\substack{p\in\mathbb{P}\\ p\hspace{0.75 mm} ramifies}}\left(\sum_{j=0}^{\infty}\frac{1}{p^{2j}}\right)
\prod_{\substack{p\in\mathbb{P}\\ p\hspace{0.75 mm} splits}}\left(\sum_{j=0}^{\infty}\frac{1}{p^{2j}}\right)^2\]
\[\leq\mathcal J\prod_{\substack{p\in\mathbb{P} \\ p\not\in
\{3, 5\}}}\left(\sum_{j=0}^\infty\frac{1}{p^{2j}}\right)^2=\mathcal J\hspace{0.75 mm}\zeta(2)^2\left(\sum_{j=0}^\infty\frac{1}{3^{2j}}\right)^{-2}\left(\sum_{j=0}^\infty
\frac{1}{5^{2j}}\right)^{-2}\]
\[=\frac{4\pi^4}{195}<2.\] 
This completes the final case. 
\end{proof}
\begin{theorem} \label{Thm2.2} 
Let $d\in K$. A number $z\in\mathcal O_{\mathbb{Q}(\sqrt{d})}\backslash\{0\}$ satisfies $I_1(z)=b\in\mathbb{Q}$ if and only if it is associated to an integer whose (traditional) abundancy index is $b$ and whose prime factors (in $\mathbb{Z}$) are all inert in $\mathcal O_{\mathbb{Q}(\sqrt{d})}$. 
\end{theorem}
\begin{proof}
Suppose $z\in\mathcal O_{\mathbb{Q}(\sqrt{d})}\backslash\{0\}$ satisfies $I_1(z)=b\in\mathbb{Q}$. If $b=1$, then the desired result is clear because $z\sim 1$. Therefore, we may assume $b>1$. Lemma \ref{Lem2.2} tells us that all primes dividing $z$ are associated to inert integer primes, which implies that $z$ is associated to an integer whose prime factors (in $\mathbb{Z}$) are all inert in $\mathcal O_{\mathbb{Q}(\sqrt{d})}$. We may, therefore, write $z\sim r$ for some $r\in\mathbb{N}$. As all primes dividing $r$ are associated to inert integer primes, we have $I_1(r)=I(r)$, where $I$ is the traditional abundancy index defined over $\mathbb{N}$. Therefore, $I(r)=I_1(r)=I_1(z)=b$. 
\par 
Conversely, if $z\sim r$, where $r$ is an integer whose (traditional) abundancy index is a rational number $b$ and whose prime factors (in $\mathbb{Z}$) are all inert in $\mathcal O_{\mathbb{Q}(\sqrt{d})}$, then $I_1(z)=I_1(r)=I(r)=b\in\mathbb{Q}$. 
\end{proof}
\begin{corollary} \label{Cor2.1}
Let $d\in K$. A number $z\in\mathcal O_{\mathbb{Q}(\sqrt{d})}\backslash\{0\}$ is a $t$-perfect number in $\mathcal O_{\mathbb{Q}(\sqrt{d})}$ if and only if it is associated to an integer that is $t$-perfect in $\mathbb{Z}$ and whose prime factors (in $\mathbb{Z}$) are all inert in $\mathcal O_{\mathbb{Q}(\sqrt{d})}$.
\end{corollary}
\begin{proof}
Simply set $b=t$ in Theorem \ref{Thm2.2}.
\end{proof}
Let us now restrict the scope of our exploration to perfect numbers in a ring $\mathcal O_{\mathbb{Q}(\sqrt{d})}$ ($d\!\in\! K$). That is, we will search for $n$-powerfully $t$-perfect numbers with $n=1$ and $t=2$. We will repeatedly make use of Corollary \ref{Cor2.1} and the following two well-known facts about perfect numbers in $\mathbb{Z}$ \cite{Nguyen00, Strayer02}.
\begin{fact} \label{Fact2.1}
A positive integer $r$ is an even perfect number (in $\mathbb{Z}$) if and only if $r=2^{p-1}(2^p-1)$ for some Mersenne prime $2^p-1$.
\end{fact}
\begin{fact} \label{Fact2.2}
If $r$ is an odd integer that is perfect in $\mathbb{Z}$ (assuming such a number exists), then $r=p^km^2$, where $p$ is an integer prime and $k,m\in\mathbb{N}$. Furthermore, $p\equiv k\equiv 1\imod{4}$, $m>1$, and $p\nmid m$ in $\mathbb{Z}$. The expression $p^km^2$ is known as the Eulerian form of the odd perfect number $r$.   
\end{fact}
\begin{theorem} \label{Thm2.3}
There are no perfect numbers in $\mathcal O_{\mathbb{Q}(\sqrt{-1})}$, the ring of Gaussian integers. 
\end{theorem} 
\begin{proof}
Suppose $z$ is perfect in $\mathcal O_{\mathbb{Q}(\sqrt{-1})}$. Then, by Corollary \ref{Cor2.1}, $z\sim r$ for some positive integer $r$ that is perfect in $\mathbb{Z}$. Furthermore, all integer primes that divide $r$ in $\mathbb{Z}$ must be inert in $\mathcal O_{\mathbb{Q}(\sqrt{-1})}$. We know (by Fact \ref{Fact1.4}) that an integer prime is inert in $\mathcal O_{\mathbb{Q}(\sqrt{-1})}$ if and only if it is congruent to $3$ modulo $4$, so we conclude that all integer primes that divide $r$ in $\mathbb{Z}$ must be congruent to 3 modulo 4. Thus, $r$ is odd, so Fact \ref{Fact2.2} tells us that there exists an integer prime $p$ that divides $r$ in $\mathbb{Z}$ and is congruent to $1$ modulo $4$. This is a contradiction, and the desired result follows. 
\end{proof}
\begin{theorem} \label{Thm2.4}
There are no perfect numbers in $\mathcal O_{\mathbb{Q}(\sqrt{-3})}$, the ring of Eisenstein integers.
\end{theorem}
\begin{proof}
Suppose $z$ is perfect in $\mathcal O_{\mathbb{Q}(\sqrt{-3})}$. Then, by Corollary \ref{Cor2.1}, $z\sim r$ for some positive integer $r$ that is perfect in $\mathbb{Z}$. Furthermore, all integer primes that divide $r$ in $\mathbb{Z}$ must be inert in $\mathcal O_{\mathbb{Q}(\sqrt{-3})}$. We know (by Fact \ref{Fact1.4}) that an integer prime is inert in $\mathcal O_{\mathbb{Q}(\sqrt{-3})}$ if and only if it is congruent to $2$ modulo $3$, so we conclude that all integer primes that divide $r$ in $\mathbb{Z}$ must be congruent to $2$ modulo $3$. If $r$ is even, then we may write $r=2^{p-1}(2^p-1)$, where $2^p-1$ is a Mersenne prime. However, then $2^p-1$ is an integer prime that divides $r$ in $\mathbb{Z}$, so we conclude $2^p-1\equiv 2\imod{3}$, which is impossible. Therefore, $r$ is odd, so we may write $r$ in the Eulerian form $r=p^km^2$. From the fact that $r$ is perfect in $\mathbb{Z}$, we have $\displaystyle{2p^km^2=\sigma(p^km^2)=\sigma(p^k)\sigma(m^2)=\sigma(m^2)\sum_{j=0}^kp^j}$. Now, Fact \ref{Fact2.2} tells us that $k$ is odd, so $p+1\vert\sum_{j=0}^kp^j$ in $\mathbb{Z}$. As $p\equiv 2\imod{3}$, we find that $\displaystyle{3\vert\sum_{j=0}^kp^j\vert 2p^km^2=2r}$ in $\mathbb{Z}$. Therefore, $3\vert r$ in $\mathbb{Z}$, which is a contradiction because all the integer primes that divide $r$ in $\mathbb{Z}$ are congruent to $2$ modulo $3$. 
\end{proof}
The method used to prove Theorems \ref{Thm2.3} and \ref{Thm2.4} may be used to explore the properties that perfect numbers in the other seven rings (corresponding to $d\in\{-163,-67,-43,-19,-11,-7,-2\}$) must possess, but the casework can become tedious very quickly. For this reason, we will only briefly explore properties of hypothetical perfect numbers in $\mathcal O_{\mathbb{Q}(\sqrt{-2})}$. 
\par 
If $z$ is perfect in $\mathcal O_{\mathbb{Q}(\sqrt{-2})}$ (a ring in which $2$ ramifies), then $z\sim r$, where $r$ is an odd integer that is perfect in $\mathbb{Z}$ and has Eulerian form $r=p^km^2$. In addition, $-2$ is a quadratic nonresidue modulo an integer prime if and only if that integer prime is congruent to $5$ or $7$ modulo $8$. Therefore, because Fact \ref{Fact2.2} states that $p\equiv 1\imod{4}$ (and because $p$ must be inert), we find that $p\equiv 5\imod{8}$. Write $m=m_1m_2$ with $m_1,m_2\in\mathbb{N}$ so that all integer primes dividing $m_1$ in $\mathbb{Z}$ are congruent to $5$ modulo $8$ and all integer primes dividing $m_2$ in $\mathbb{Z}$ are congruent to $7$ modulo $8$. Let $\displaystyle{m_1=\prod_{j=1}^sq_j^{\alpha_j}}$  be the canonical prime factorization of $m_1$ in $\mathbb{Z}$, and let $L=\vert\{j\in\{1,2,\ldots,s\}: \alpha_j\hspace{1.5 mm}is\hspace{1.5 mm} odd\}\vert$. Using this notation, we may state and prove the following theorem. 
\begin{theorem} \label{Thm2.5}
Let all notation be as in the preceding paragraph. If $k\equiv 1\imod{8}$, then $L$ is odd. If $k\equiv 5\imod{8}$, then $L$ is even.
\end{theorem} 
\begin{proof}
First, as $p$, $k$, and $m$ are all odd, we have $p^km^2=p\cdot (p^2)^{\frac{k-1}{2}}m^2\equiv p(1)(1)\equiv 5\imod{8}$. Therefore, $2r=2p^km^2\equiv 2\imod{8}$. Now,
\[2r=\sigma(p^km^2)=\sigma(p^k)\sigma(m_1^2)\sigma(m_2^2)=\sigma(m_2^2)\left(\sum_{l=0}^kp^l\right)\prod_{j=1}^s\left(\sum_{l=0}^{2\alpha_j}q_j^l\right).\]
Let $q$ be an integer prime that divides $m_2$ in $\mathbb{Z}$. Then $\sigma(q^{\upsilon_q(m_2^2)})=1+q+q^2+\cdots+q^{\upsilon_q(m_2^2)}\equiv 1\imod{8}$ because $\upsilon_q(m_2^2)$ is even and $q\equiv 7\imod{8}$. This implies that $\sigma(m_2^2)\equiv 1\imod{8}$. Furthermore, for any $j\in\{1,2,\ldots,s\}$, we have $\displaystyle{\sum_{l=0}^{2\alpha_j}q_j^l\equiv\sum_{l=0}^{2\alpha_j}5^l\equiv 6\alpha_j+1\imod{8}}$. 
Assume $k\equiv 1\imod{8}$. One may verify that, under this assumption, $\displaystyle{\sum_{l=0}^kp^l\equiv\sum_{l=0}^k 5^l\equiv 6\imod{8}}$. We then have
\[2\equiv 2r\equiv\sigma(m_2^2)\left(\sum_{l=0}^kp^l\right)\prod_{j=1}^s\left(\sum_{l=0}^{2\alpha_j}q_j^l\right)\equiv 6\prod_{j=1}^s(6\alpha_j+1)\imod{8},\]
which implies $\displaystyle{\prod_{j=1}^s(6\alpha_j+1)\equiv 3\imod{4}}$. Whenever $\alpha_j$ is even, $6\alpha_j+1\equiv 1\imod{4}$. In addition, whenever $\alpha_j$ is odd, $6\alpha_j+1\equiv 3\imod{4}$. Therefore, $L$, which is the number of integers $j\in\{1,2,\ldots,s\}$ such that $\alpha_j$ is odd, must be an odd number. 
\par 
On the other hand, if $k\equiv 5\imod{8}$, then $\displaystyle{\sum_{l=0}^kp^l\equiv\sum_{l=0}^k 5^l\equiv 2\imod{8}}$. We then have
\[2\equiv 2r\equiv\sigma(m_2^2)\left(\sum_{l=0}^kp^l\right)\prod_{j=1}^s\left(\sum_{l=0}^{2\alpha_j}q_j^l\right)\equiv 2\prod_{j=1}^s(6\alpha_j+1)\imod{8},\]
which implies $\displaystyle{\prod_{j=1}^s(6\alpha_j+1)\equiv 1\imod{4}}$. Again, whenever $\alpha_j$ is even, $6\alpha_j+1\equiv 1\imod{4}$. Also, whenever $\alpha_j$ is odd, $6\alpha_j+1\equiv 3\imod{4}$. Therefore, $L$ must be an even number in this case. 
\end{proof}
We note that it has been conjectured (supposedly by Descartes) that the value of $k$ in the Eulerian form of any hypothetical odd number that is perfect in $\mathbb{Z}$ must be $1$. If the conjecture is true, then Theorem \ref{Thm2.5} implies that $L$ must be odd.
\par 
We note that there are definitely rings $\mathcal O_{\mathbb{Q}(\sqrt{d})}$ with $d\in K$ that contain perfect numbers. For example, one may show that an integer prime is inert in $\mathcal O_{\mathbb{Q}(\sqrt{-11})}$ if and only if that integer prime is congruent to $2$, $6$, $7$, $8$, or $10$ modulo $11$. Therefore, if $2^p-1$ is a Mersenne prime that is congruent to $6$ or $7$ (one may show, using the fact that $p$ is prime, that $2^p-1$ cannot be congruent to $2$, $8$, or $10$) modulo $11$, then $2^{p-1}(2^p-1)$ is perfect in $\mathcal O_{\mathbb{Q}(\sqrt{-11})}$. For example, $28$, $8\hspace{0.25 mm}128$, $2^{13-1}(2^{13}-1)$, and $2^{17-1}(2^{17}-1)$ (this list is not exhaustive) are all perfect in $\mathcal O_{\mathbb{Q}(\sqrt{-11})}$.   

\section{Suggestions for Further Exploration} 
We acknowledge the entirely possible generalization of the definitions presented here to the other quadratic integer rings. In particular, generalizing the abundancy index to unique factorization domains $\mathcal O_{\mathbb{Q}(\sqrt{d})}$ with $d>0$ seems to be a manageable task. 
\par 
Even if we continue to restrict our attention to the rings $\mathcal O_{\mathbb{Q}(\sqrt{d})}$ with $d\!\in\! K$, we may ask some interesting questions. For example, in a given ring $\mathcal O_{\mathbb{Q}(\sqrt{d})}$ with $d\in K$, one may wish to examine the properties of $t$-perfect numbers for $t>2$. 
\section{Acknowledgments} 
The author would like to thank Professor Pete Johnson for inviting him to the 2014 REU Program in Algebra and Discrete Mathematics at Auburn University. The author would also like to thank the unknown referee for his or her careful reading.

\end{document}